\documentclass[12pt,twoside]{amsart}
\usepackage{amsmath, amsthm, amscd, amsfonts, amssymb, graphicx}
\usepackage[bookmarksnumbered, plainpages]{hyperref}
\usepackage[latin1]{inputenc}
\newtheorem{thm}{Theorem}[section]
\newtheorem{cor}[thm]{Corollary}
\newtheorem{lem}[thm]{Lemma}
\newtheorem{prop}[thm]{Proposition}
\newtheorem{defn}[thm]{Definition}

\numberwithin{equation}{section}

\allowdisplaybreaks

\begin{document}

\title{\bf Affine connections on the algebra of differential forms}

\author{Yong Wang$^*$, Shuang Wang}

\thanks{{\scriptsize
\hskip -0.4 true cm \textit{2010 Mathematics Subject Classification:}
53C40; 53C42.
\newline \textit{Corresponding author:} Yong Wang}}

\maketitle

\begin{abstract}
 In this paper, we define the semi-symmetric metric
 connection on the algebra of differential forms. We compute some special semi-symmetric metric
 connections and their curvature tensor and their Ricci tensor on the algebra of differential forms. We study the distribution
  on the algebra of differential forms and we get its Gauss-Codazzi-Ricci equations associated to the semi-symmetric metric
 connection. We also study the Lie derivative of the distribution
  on the algebra of differential forms. We define the canonical connection and the Schouten connection and the Vrancreanu connection
  on the algebra of differential forms and get some properties of these connections.
\end{abstract}

\vskip 0.2 true cm


\pagestyle{myheadings}
\markboth{\rightline {\scriptsize Wang}}
         {\leftline{\scriptsize Affine connections on the algebra of differential forms}}

\bigskip
\bigskip


\section{ Introduction}
 \indent  The definition of a semi-symmetric metric connection was given by H. Hayden in \cite{Ha}. In 1970, K. Yano \cite{Ya}
        considered a semi-symmetric metric connection and studied some of its properties. He proved that a Riemannian manifold admitting
         the semi-symmetric metric connection has vanishing curvature tensor if and only if it is conformally flat. Motivated by the Yano'
         result,
         in \cite{SO}, Sular and \"{O}zgur studied warped product manifolds with a
         semi-symmetric metric connection, they computed curvature of semi-symmetric metric connection
          and considered Einstein warped product manifolds with a semi-symmetric metric connection. In \cite{Wa1},
          we studied non integrable distributions in a Riemannian manifold with a semi-symmetric metric connection,
           a semi-symmetric non-metric connection and a statistical connection. We obtained
            the Gauss, Codazzi, and Ricci equations for non integrable distributions with respect to
             the semi-symmetric metric connection, the semi-symmetric non-metric connection and the statistical connection. We also gave some applications of our Gauss-Codazzi-Ricci equations.
 In \cite{Va}, Van stated the Lie derivative of normal connection on a submanifold of the Riemannian
manifold. He introduced the Lie derivative of the normal curvature tensor on the submanifold and gave
some relations between the normal curvature tensor on the submanifold and curvature tensor on the
ambient manifold in the sense of the Lie derivative of normal connection. In \cite{YE}, Youssef and Elsayed provided a global investigation of the geometry of parallelizable manifolds (or absolute parallelism geometry) frequently used for application. They discussed the different
linear connections and curvature tensors from a global point of view. They gave an existence and
uniqueness theorem for a remarkable linear connection, called the canonical connection.\\
\indent
 In \cite{BG}, Bruce and Grabowski examined the notion of a Riemannian $\mathbb{Z}_2^n$ manifold.
 They showed that the basic notions and tenets of Riemannian geometry directly generalized to the setting of $\mathbb{Z}_2^n$-geometry.
For example, the fundamental theorem holded in the higher graded setting. They pointed out the
similarities and differences with Riemannian supergeometry. In \cite{Wa2}, we defined the semi-symmetric metric
 connection on super Riemannian manifolds. We computed the semi-symmetric metric connection and its curvature tensor
 and its Ricci tensor on super warped product spaces. We introduced two kind of super warped product spaces
 with the semi-symmetric metric connection and gave the conditions under which these two super warped product spaces
  with the semi-symmetric metric connection were the Einstein super spaces with the semi-symmetric metric connection.
 In \cite{MS}, Monterde and Sanchez characterized the homogeneous Riemannian graded metrics on the algebra
of differential forms for which the exterior derivative was a Killing
graded vector field. It was shown that all of them were odd,
and were naturally associated to an underlying smooth Riemannian metric.
It was also shown that all of them were Ricci-flat in the graded sense, and
had a graded Laplacian operator that annihilates the whole algebra of
differential forms. In this paper, we define the semi-symmetric metric
 connection on the algebra of differential forms. We compute some special semi-symmetric metric
 connections and their curvature tensor and their Ricci tensor on the algebra of differential forms. We study the distribution
  on the algebra of differential forms and we get its Gauss-Codazzi-Ricci equations associated to the semi-symmetric metric
 connection. We also study the Lie derivative of the distribution
  on the algebra of differential forms. We define the canonical connection and the Schouten connection and the Vrancreanu connection
  on the algebra of differential forms and get some properties of these connections.
\\
  \indent In Section 2, we define the semi-symmetric metric
 connection on the algebra of differential forms. We compute some special semi-symmetric metric
 connections and their curvature tensor and their Ricci tensor on the algebra of differential forms. In Section 3, we study the distribution
  on the algebra of differential forms and we get its Gauss-Codazzi-Ricci equations associated to the semi-symmetric metric
 connection.
  In Section 4, we also study the Lie derivative of the distribution
  on the algebra of differential forms. In Section 5, we define the canonical connection and the Schouten connection and the Vrancreanu connection
  on the algebra of differential forms and get some properties of these connections.


\vskip 1 true cm

\section{A semi-symmetric metric
 connection on on the algebra of differential forms}

 Let $M$ be an $m$-dimensional smooth manifold and $\Omega(M)$ be its corresponding $\mathbb{Z}_2$-graded commutative $\mathbb{R}$-algebra of differential forms, then the pair $(M,\Omega(M))$ is an $(m,m)$-dimensional $\mathbb{Z}_2$-graded manifold. Let ${\rm Der}\Omega(M)$ be the left graded $\Omega(M)$-module of all derivations on $\Omega(M)$. We denote the elements in ${\rm Der}\Omega(M)$ by $X,Y,Z,\cdots$ and denote the smooth vector fields on $M$ by $\overline{X},\overline{Y},\overline{Z},\cdots$.
 ${\rm Der}\Omega(M)$ is a graded Lie algebra with the usual graded commutator  It can also be regarded as a right graded $\Omega(M)$-module with multiplication $X\alpha=(-1)^{|\alpha||X|}\alpha X.$   Let ${\rm Hom}({\rm Der}\Omega(M),\Omega(M))$ be the right graded $\Omega(M)$-module of $\Omega(M)$-linear
graded homomorphisms from the derivations ${\rm Der}\Omega(M)$ into the superfunctions $\Omega(M)$. This is the module of graded differential $1$-forms on $(M,\Omega(M))$.
   \begin{defn}
   (\cite{MS})
 A graded metric on the algebra of differential forms is a graded
symmetric, non-degenerate, bilinear map
$G:{\rm Der}\Omega(M)\times {\rm Der}\Omega(M)\rightarrow \Omega(M); (X,Y)\rightarrow\left<X,Y;G\right>.$
That is, a map satisfying the following conditions:\\
(1) $\left<X,Y;G\right>=(-1)^{|X||Y|}\left<Y,X;G\right>$,\\
(2) $\left<\alpha X,Y;G\right>=\alpha\left<X,Y;G\right>=(-1)^{|X||\alpha|}\left<X,\alpha Y;G\right>$, $\alpha\in \Omega(M)$,\\
(3) The linear map $X\mapsto \left<\cdot,X; G \right>$ is an isomorphism between the $\Omega(M)$-
modules ${\rm Der}\Omega(M)$ and ${\rm Hom}({\rm Der}\Omega(M),\Omega(M))$.\\
\end{defn}
A graded metric is homogeneous of degree $k\in \mathbb{Z}$ if $|\left<X,Y;G\right>|=|X|+|Y|+k.$
A graded metric is even (resp. odd) if  $|\left<X,Y;G\right>|=|X|+|Y|({\rm mod}~~ 2)$
(resp.$|\left<X,Y;G\right>|=|X|+|Y|+1({\rm mod} ~~2)$).
Let $U$ be an open coordinate neighborhood in $M$.
Let $\left\{\overline{X_1},\cdots,\overline{X_m}\right\}$ be
a local frame of vector fields in $U$. It is easy to check that $\left\{L_{\overline{X_1}},\cdots,L_{\overline{X_m}},i_{\overline{X_1}},\cdots,i_{\overline{X_m}}\right\}$ is a local frame for ${\rm Der}\Omega(U)$(cf. \cite{FN}). Thus, a graded metric
is completely determined by its action on the pairs of derivations $(L_{\overline{X}},L_{\overline{Y}})$,$(L_{\overline{X}},i_{\overline{Y}})$, and $(i_{\overline{X}},i_{\overline{Y}})$
where $\overline{X}$ and $\overline{Y}$ are vector fields on $M$.
\begin{defn}(Definition 4.1 in \cite{MS}) An graded connection on $\Omega(M)$ is a $\mathbb{Z}_2$-degree preserving map\\
$$\nabla:~~ {\rm Der}\Omega(M)\times {\rm Der}\Omega(M)\rightarrow {\rm Der}\Omega(M);~~(X,Y)\mapsto \nabla_XY,$$
that satisfies the following\\
1) Bi-linearity $$\nabla_X(Y+Z)=\nabla_XY+\nabla_XZ;~~\nabla_{X+Y}Z=\nabla_XZ+\nabla_YZ,$$
2)$\Omega(M)$-linearrity in the first argument
$$\nabla_{\alpha X}Y=\alpha\nabla_XY,$$
3)The Leibniz rule
$$\nabla_X(\alpha Y)=X(\alpha)Y+(-1)^{|X||\alpha|}\alpha\nabla_XY,$$
for all homogeneous $X,Y,Z\in {\rm Der}\Omega(M)$ and $\alpha\in \Omega(M)$.
\end{defn}
\begin{defn}(\cite{MS})
 The torsion tensor of an affine connection \\
 $T_\nabla:~~{\rm Der}\Omega(M)\otimes_{\Omega(M)} {\rm Der}\Omega(M)\rightarrow {\rm Der}\Omega(M)$
 is
defined as
$$T_\nabla(X,Y):=\nabla_XY-(-1)^{|X||Y|}\nabla_YX-[X,Y],$$
for any (homogeneous) $X, Y \in {\rm Der}\Omega(M)$. An affine connection is said to be symmetric if the torsion vanishes.
\end{defn}
\begin{defn}(\cite{MS})
Let $G$ be a graded metric on $\Omega(M)$ and $\nabla$ an even graded connection. $\nabla$ is said to be metric compatible if
and only if
$$ X\left<Y,Z;G\right>=\left<\nabla_XY,Z;G\right>+(-1)^{|X||Y|}\left<Y,\nabla_XZ;G\right>,$$
for any $X, Y, Z\in {\rm Der}\Omega(M)$.
\end{defn}
\begin{thm}(Theorem 4.2 in \cite{MS})There is a unique symmetric (torsionless) and metric compatible
affine graded connection $\nabla^L$ which satisfies the Koszul formula
\begin{align}
2\left<\nabla^L_XY,Z;G\right>&=X\left<Y,Z;G\right>+\left<[X,Y],Z;G\right>\\
&+(-1)^{|X|(|Y|+|Z|)}(Y\left<Z,X;G\right>-\left<[Y,Z],X;G\right>)\notag\\
&-(-1)^{|Z|(|X|+|Y|)}(Z\left<X,Y;G\right>-\left<[Z,X],Y;G\right>),\notag
\end{align}
for all homogeneous $X,Y,Z\in {\rm Der}\Omega(M)$.
\end{thm}
Given a Riemannian metric $g$ on $M$ its corresponding graded metric $G_g$ is given by
$$\left<L_{\overline{X}},L_{\overline{Y}};G_g\right>=d(g(\overline{X},\overline{Y})),$$
$$\left<L_{\overline{X}},i_{\overline{Y}};G_g\right>=\left<i_{\overline{X}},L_{\overline{Y}};G_g\right>=g(\overline{X},\overline{Y}),$$
$$\left<i_{\overline{X}},i_{\overline{Y}};G_g\right>=0,$$
for the vector fields $\overline{X}$,$\overline{Y}$ on $M$.
\begin{defn}(\cite{MS})
The Riemannian curvature tensor of an affine graded connection
is defined as
\begin{equation}
R_\nabla(X, Y )Z =\nabla_X\nabla_Y-(-1)^{|X||Y|}\nabla_Y\nabla_X-\nabla_{[X,Y]}Z,
\end{equation}
for all $X, Y$ and $Z \in {\rm Der}\Omega(M)$.
\end{defn}
Let $\{\overline{X_k}\}_{k=1}^m$ be an orthonormal frame for $g$. Then $\{L_{\overline{X_k}},i_{\overline{X_k}}\}_{k=1}^m$ is a basis of graded
derivations that satisfies the following relations:
\begin{equation}
\left<L_{\overline{X_k}},L_{\overline{X_l}}\right>=0=\left<i_{\overline{X_k}},i_{\overline{X_l}};G_g\right>,~~~
\left<L_{\overline{X_k}},i_{\overline{X_l}};G_g\right>=\left<i_{\overline{X_l}},L_{\overline{X_k}};G_g\right>=\delta_{kl}.
\end{equation}
\begin{defn}(\cite{MS})
 The Ricci curvature tensor of an affine connection is the rank-$2$ covariant tensor
defined as
 \begin{equation}
Ric_\nabla(X, Y ):=\sum_{k=1}^m\left<R_\nabla(L_{\overline{X_k}},X)Y,i_{\overline{X_k}},G_g\right>-(-1)^{|X|+|Y|}
\sum_{l=1}^m\left<R_\nabla(i_{\overline{X_l}},X)Y,L_{\overline{X_l}},G_g\right>
\end{equation}
where $X,Y\in {\rm Der}\Omega(M)$.
\end{defn}
\begin{defn}(\cite{Wa2})Let $P\in {\rm Der}\Omega(M)$ which satisfied $|G|+|P|=0$
and the semi-symmetric metric connection ${\nabla}$ on $(M, \Omega(M),G)$
is given by
\begin{equation}
{\nabla}_XY=\nabla^L_XY+X\cdot \left<Y,P;G\right>-\left<X,Y;G\right>P,
\end{equation}
for any homogenous $X,Y\in {\rm Der}\Omega(M)$.
\end{defn}
We have
\begin{equation}
T_\nabla(X,Y)=X\cdot \left<Y,P;G\right>-(-1)^{|X||Y|}Y\cdot \left<X,P;G\right>,
\end{equation}
In this case, we call that ${\nabla}_XY$ is a semi-symmetric connection. By (2.10) in \cite{Wa2}
then $\nabla$ preserves the metric.
\begin{thm}(Theorem 2.14 in \cite{Wa2})There is a unique semi-symmetric metric compatible
affine connection $\nabla$ on $(M, \Omega(M),G)$.
\end{thm}
We also denote $\left<X,Y;G\right>$ by $G(X,Y)$.
\begin{prop}(Proposition 2.15 in \cite{Wa2})The following equality holds
\begin{align}
&R_\nabla(X,Y)Z=R^L(X,Y)Z+(-1)^{(|X|+|Y|)|Z|}\left<Z,\nabla^L_XP;G\right>)Y\\
&-(-1)^{|X||Y|}(-1)^{(|X|+|Y|)|Z|}\left<Z,\nabla^L_YP;G\right>X
-(-1)^{|G(Y,Z)||X|}G(Y,Z)\nabla^L_XP\notag\\
&+(-1)^{|X||Y|}(-1)^{|G(X,Z)||Y|}G(X,Z)\nabla^L_YP
+(-1)^{|X|(|Y|+|Z|)}(-1)^{|Y||Z|}G(Z,P)G(Y,P)X\notag\\
&-(-1)^{|X|(|Y|+|Z|)}G(Y,Z)G(P,P)X
-(-1)^{(|X|+|Y|)|Z|}G(Z,P)G(X,P)Y\notag\\
&+(-1)^{|Y||Z|}G(X,Z)G(P,P)Y
+(-1)^{|X||G(Y,Z)|}G(Y,Z)G(X,P)P\notag\\
&-(-1)^{|X||Y|}(-1)^{|Y||G(X,Z)|}G(X,Z)G(Y,P)P.\notag
\end{align}
\end{prop}
Let $\nabla^g$ be the Levi-Civita connection on $(M,g)$. By (2.5) and
\begin{equation}
\nabla^L_{L_{\overline{X}}}L_{\overline{Y}}=L_{\nabla^g_{\overline{X}}{\overline{Y}}},
~~\nabla^L_{L_{\overline{X}}}i_{\overline{Y}}=i_{\nabla^g_{\overline{X}}{\overline{Y}}},~~
\nabla^L_{i_{\overline{X}}}L_{\overline{Y}}=i_{\nabla^g_{\overline{X}}{\overline{Y}}}~~
\nabla^L_{i_{\overline{X}}}i_{\overline{Y}}=0,
\end{equation} we have
\begin{prop}Let $P=i_U$ for $U\in \Gamma(M,TM)$ and $G=G_g$ for a Riemannian metric $g$ on $M$, the following equalities hold
\begin{align}
&\nabla_{L_{\overline{X}}}L_{\overline{Y}}=L_{\nabla^g_{\overline{X}}{\overline{Y}}}+g(\overline{Y},U)L_{\overline{X}}-d(g(\overline{X},
\overline{Y}))i_U,\\
&~~\nabla_{L_{\overline{X}}}i_{\overline{Y}}=i_{\nabla^g_{\overline{X}}{\overline{Y}}}-g(\overline{X},\overline{Y})i_U,\notag\\
&\nabla_{i_{\overline{X}}}L_{\overline{Y}}=i_{\nabla^g_{\overline{X}}{\overline{Y}}}
+g(\overline{Y},U)i_{\overline{X}}-g(\overline{X},
\overline{Y})i_U,\notag\\
&\nabla_{i_{\overline{X}}}i_{\overline{Y}}=0.\notag
\end{align}
\end{prop}
\begin{prop}Let $P=\omega L_U$ for $U\in \Gamma(M,TM)$ and $\omega\in \Omega^{{\rm odd}}(M)$ and $G=G_g$ for a Riemannian metric $g$ on $M$. In this case, we write $\nabla'$ instead of $\nabla$ and the following equalities hold
\begin{align}
&\nabla'_{L_{\overline{X}}}L_{\overline{Y}}=L_{\nabla^g_{\overline{X}}{\overline{Y}}}+\omega d(g(\overline{Y},U))L_{\overline{X}}-d(g(\overline{X},
\overline{Y}))\omega L_U,\\
&~~\nabla'_{L_{\overline{X}}}i_{\overline{Y}}=i_{\nabla^g_{\overline{X}}{\overline{Y}}}-\omega g(\overline{Y},U) L_{\overline{X}}-g(\overline{X},\overline{Y})\omega L_U,\notag\\
&\nabla'_{i_{\overline{X}}}L_{\overline{Y}}=i_{\nabla^g_{\overline{X}}{\overline{Y}}}
+\omega d(g(\overline{Y},U))i_{\overline{X}}-g(\overline{X},
\overline{Y})\omega L_U,\notag\\
&\nabla'_{i_{\overline{X}}}i_{\overline{Y}}=\omega g(\overline{Y},U)i_{\overline{X}} .\notag
\end{align}
\end{prop}
By p. 165 in \cite{MS}, the graded curvature of $G_g$ is given by
\begin{align}
&R^L(L_{\overline{X}},L_{\overline{Y}})L_{\overline{Z}}=L_{R^g(\overline{X},\overline{Y})\overline{Z}},\\
&R^L(L_{\overline{X}},L_{\overline{Y}})i_{\overline{Z}}=i_{R^g(\overline{X},\overline{Y})\overline{Z}},~~
R^L(L_{\overline{X}},i_{\overline{Y}})L_{\overline{Z}}=i_{R^g(\overline{X},\overline{Y})\overline{Z}},\notag\\
&R^L(L_{\overline{X}},i_{\overline{Y}})i_{\overline{Z}}=0,~~
R^L(i_{\overline{X}},i_{\overline{Y}})=0.\notag
\end{align}
By (2.7) and (2.11), we have
\begin{prop}Let $P=i_U$ for $U\in \Gamma(M,TM)$ and $G=G_g$ for a Riemannian metric $g$ on $M$, the following equalities hold
\begin{align}
&R(L_{\overline{X}},L_{\overline{Y}})L_{\overline{Z}}=L_{R^g(\overline{X},\overline{Y})\overline{Z}}
+g(\overline{Z},\nabla^g_{\overline{X}}U)L_{\overline{Y}}-g(\overline{Z},\nabla^g_{\overline{Y}}U)L_{\overline{X}}\\
&-d(g(\overline{Y},\overline{Z}))i_{\nabla^g_{\overline{X}}U}+d(g(\overline{X},\overline{Z}))i_{\nabla^g_{\overline{Y}}U}+g(\overline{Z},U)g(\overline{Y},U)L_{\overline{X}}\notag\\
&-g(\overline{Z},U)g(\overline{X},U)L_{\overline{Y}}+d(g(\overline{Y},\overline{Z}))g(\overline{X},U)i_U-
d(g(\overline{X},\overline{Z}))g(\overline{Y},U)i_U,\notag\\
&R(L_{\overline{X}},L_{\overline{Y}})i_{\overline{Z}}=i_{R^g(\overline{X},\overline{Y})\overline{Z}}-g(\overline{Y},\overline{Z})i_{
\nabla^g_{\overline{X}}U}
+g(\overline{X},\overline{Z})i_{
\nabla^g_{\overline{Y}}U}\notag\\
&+g(\overline{Y},\overline{Z})g(\overline{X},U)i_U-g(\overline{X},\overline{Z})g(\overline{Y},U)i_U,\notag\\
&R(L_{\overline{X}},i_{\overline{Y}})L_{\overline{Z}}=i_{R^g(\overline{X},\overline{Y})\overline{Z}}+
g(\overline{Z},\nabla^g_{\overline{X}}U)i_{\overline{Y}}-g(\overline{Y},\overline{Z})i_{\nabla^g_{\overline{X}}U}\notag\\
&-g(\overline{Z},U)g(\overline{X},U)i_{\overline{Y}}+g(\overline{Y},\overline{Z})g(\overline{X},U)i_U
,\notag\\
&R(L_{\overline{X}},i_{\overline{Y}})i_{\overline{Z}}=0,~~
R(i_{\overline{X}},i_{\overline{Y}})=0.\notag
\end{align}
\end{prop}
By (2.4) and (2.12), we can get
\begin{thm}Let $P=i_U$ for $U\in \Gamma(M,TM)$ and $G=G_g$ for a Riemannian metric $g$ on $M$, then $(M,\Omega(M),\nabla)$ is Ricci flat. When $U=0$, we get Theorem 4.3 in \cite{MS}.
\end{thm}
By (2.7) and (2.11), we have
\begin{prop}Let $P=\omega L_U$ for $U\in \Gamma(M,TM)$ and $\omega\in \Omega^{{\rm odd}}(M)$ and $G=G_g$ for a Riemannian metric $g$ on $M$. The following equalities hold
\begin{align}
 &R'(L_{\overline{X}},L_{\overline{Y}})L_{\overline{Z}}=L_{R^g(\overline{X},\overline{Y})\overline{Z}}
+L_{\overline{X}}(\omega)d(g(\overline{Z},U))L_{\overline{Y}}+\omega d(g(\overline{Z},\nabla^g_{\overline{X}}U))L_{\overline{Y}}\\
&-L_{\overline{Y}}(\omega)d(g(\overline{Z},U))L_{\overline{X}}-\omega d(g(\overline{Z},\nabla^g_{\overline{Y}}U))L_{\overline{X}}
-d(g(\overline{Y},\overline{Z}))L_{\overline{X}}(\omega)L_U\notag\\
&-d(g(\overline{Y},\overline{Z}))\omega L_{\nabla^g_{\overline{X}}U}
+d(g(\overline{X},\overline{Z}))L_{\overline{Y}}(\omega)L_U+d(g(\overline{X},\overline{Z}))\omega L_{\nabla^g_{\overline{Y}}U},\notag\\
 &R'(L_{\overline{X}},L_{\overline{Y}})i_{\overline{Z}}=i_{R^g(\overline{X},\overline{Y})\overline{Z}}
-L_{\overline{X}}(\omega)g(\overline{Z},U)L_{\overline{Y}}-\omega g(\overline{Z},\nabla^g_{\overline{X}}U)L_{\overline{Y}}\notag\\
&+L_{\overline{Y}}(\omega)g(\overline{Z},U)L_{\overline{X}}+\omega g(\overline{Z},\nabla^g_{\overline{Y}}U)L_{\overline{X}}
-g(\overline{Y},\overline{Z})L_{\overline{X}}(\omega)L_U\notag\\
&-g(\overline{Y},\overline{Z})\omega L_{\nabla^g_{\overline{X}}U}
+g(\overline{X},\overline{Z})L_{\overline{Y}}(\omega)L_U+g(\overline{X},\overline{Z})\omega L_{\nabla^g_{\overline{Y}}U},\notag\\
 &R'(L_{\overline{X}},i_{\overline{Y}})L_{\overline{Z}}=i_{R^g(\overline{X},\overline{Y})\overline{Z}}+
L_{\overline{X}}(\omega)d(g(\overline{Z},U))i_{\overline{Y}}+\omega d(g(\overline{Z},\nabla^g_{\overline{X}}U))i_{\overline{Y}}\notag\\
&-i_{\overline{Y}}(\omega)d(g(\overline{Z},U))L_{\overline{X}}+\omega g(\overline{Z},\nabla^g_{\overline{Y}}U)L_{\overline{X}}
-g(\overline{Y},\overline{Z})L_{\overline{X}}(\omega)L_U\notag\\
&-g(\overline{Y},\overline{Z})\omega L_{\nabla^g_{\overline{X}}U}
-d(g(\overline{X},\overline{Z}))i_{\overline{Y}}(\omega)L_U+d(g(\overline{X},\overline{Z}))\omega i_{\nabla^g_{\overline{Y}}U}
,\notag\\
 &R'(L_{\overline{X}},i_{\overline{Y}})i_{\overline{Z}}=L_{\overline{X}}(\omega)g(\overline{Z},U)i_{\overline{Y}}
+\omega g(\overline{Z},\nabla^g_{\overline{X}}U)i_{\overline{Y}}+i_{\overline{Y}}(\omega)g(\overline{Z},U)L_{\overline{X}}\notag\\
&+g(\overline{X},\overline{Z})i_{\overline{Y}}(\omega)L_U-g(\overline{X},\overline{Z})\omega i_{\nabla^g_{\overline{Y}}U},\notag\\
 &R'(i_{\overline{X}},i_{\overline{Y}})L_{\overline{Z}}=i_{\overline{X}}(\omega)d(g(\overline{Z},U))i_{\overline{Y}}
-\omega g(\overline{Z},\nabla^g_{\overline{X}}U)i_{\overline{Y}}
+i_{\overline{Y}}(\omega)d(g(\overline{Z},U))i_{\overline{X}}\notag\\
&
-\omega g(\overline{Z},\nabla^g_{\overline{Y}}U)i_{\overline{X}}
-g(\overline{Y},\overline{Z})i_{\overline{X}}(\omega)L_U+g(\overline{Y},\overline{Z})\omega i_{\nabla^g_{\overline{X}}U}\notag\\
&-g(\overline{X},\overline{Z})i_{\overline{Y}}(\omega)L_U+g(\overline{X},\overline{Z})\omega i_{\nabla^g_{\overline{Y}}U},\notag\\
 &R'(i_{\overline{X}},i_{\overline{Y}})i_{\overline{Z}}=i_{\overline{X}}(\omega)g(\overline{Z},U)i_{\overline{Y}}
+i_{\overline{Y}}(\omega)g(\overline{Z},U)i_{\overline{X}}.\notag
\end{align}
\end{prop}
By (2.4) and (2.13), we have
\begin{thm}Let $P=\omega L_U$ for $U\in \Gamma(M,TM)$ and $\omega\in \Omega^{{\rm odd}}(M)$ and $G=G_g$ for a Riemannian metric $g$ on $M$, then $(M,\Omega(M),\nabla')$ is Einstein with the coefficient $-L_U\omega$, i.e. ${\rm Ric}'=-G_g\cdot L_U\omega.$
\end{thm}

\section{The distribution
  on the algebra of differential forms and a semi-symmetric metric connection}
\indent Let $D\subseteq {\rm Der}(\Omega(M))$ and $D^\bot\subseteq {\rm Der}(\Omega(M))$ be two distributions on ${\rm Der}(\Omega(M))$;
 that is, two $\Omega(M)$-submodules of ${\rm Der}(\Omega(M))$. Let ${\rm Der}(\Omega(M))=D\oplus D^\bot.$ and
 $G=G^D\oplus G^{D^\bot}$ where $G^D$ and $ G^{D^\bot}$ are metric on $D$ and $D^\bot$ respectively.\\
\indent {\bf Example}. Let $(M,\overline{X_1},\cdots,\overline{X_m})$ be a parallelizable manifold (see the definition 1.1 in \cite{YE}). Let
$D=\{L_{\overline{X_1}},\cdots,L_{\overline{X_k}},i_{\overline{X_1}},\cdots,i_{\overline{X_k}}\}$ and
$D^\bot=\{L_{\overline{X_{k+1}}},\cdots,L_{\overline{X_m}},i_{\overline{X_{K+1}}},\cdots,i_{\overline{X_m}}\}$ be two distributions
and ${\rm Der}(\Omega(M))=D\oplus D^\bot.$ Another example is the warped product (similar to the definition 5 in \cite{BG}).\\
\indent Let $\pi^D:{\rm Der}(\Omega(M))\rightarrow D$, $\pi^{D^\bot}:{\rm Der}(\Omega(M))\rightarrow D^\bot$ be the projections with the even grading.
For $X,Y \in D$, we define $\nabla^{D,L}_XY=\pi^D(\nabla^L_XY)$ and $[X,Y]^D=\pi^D([X,Y])$ and $[X,Y]^{D^\bot}=\pi^{D^\bot}([X,Y])$.
we have for $X,Y \in  D$ and $\alpha\in \Omega(M)$
\begin{equation}
\nabla^{D,L}_{\alpha X}Y=\alpha\nabla^{D,L}_{X}Y,~~\nabla^{D,L}_{X}(\alpha Y)=X(\alpha)Y+(-1)^{|X||\alpha|}\alpha\nabla^{D,L}_{X}Y,
\end{equation}
\begin{equation}
\nabla^{D,L}_XG^D=0,~~~~T^{D,L}(X,Y):=\nabla^{D,L}_{X}Y-(-1)^{|X||Y|}\nabla^{D,L}_{Y}X-[X,Y]=-[X,Y]^{D^\bot},
\end{equation}
and
\begin{equation}
\nabla^L_XY=\nabla^{D,L}_{X}Y+B(X,Y),~~B(X,Y)=\pi^{D^\bot}\nabla^L_XY,
\end{equation}
\begin{equation}
B(\alpha X,Y)=\alpha B(X,Y),~~B(X,\alpha Y)=(-1)^{|X||\alpha|}\alpha B(X,Y);
\end{equation}
\begin{equation}
~~B(X,Y)=(-1)^{|X||Y|}B(Y,X)+[X,Y]^{D^\bot}.
\end{equation}
By (3.2), we have
\begin{align}
2\left<\nabla^{D,L}_XY,Z;G^D\right>&=X\left<Y,Z;G^D\right>+\left<[X,Y]^D,Z;G^D\right>\\
&+(-1)^{|X|(|Y|+|Z|)}(Y\left<Z,X;G^D\right>-\left<[Y,Z]^D,X;G^D\right>)\notag\\
&-(-1)^{|Z|(|X|+|Y|)}(Z\left<X,Y;G^D\right>-\left<[Z,X]^D,Y;G^D\right>),\notag
\end{align}
for all homogeneous $X,Y,Z\in D$. Similarly to Theorem 2.5, we have
\begin{thm}
There exists a unique partial linear connection ${\nabla}^{L,D}: D\times D\rightarrow D$ on $D$, which satisfies the property (3.2).
\end{thm}
In the definition 2.8, we use $U$ instead of $P$.
Let $U^D=\pi^DU$ and $U^{D^\bot}=\pi^{D^\bot}U$, then $U=U^D+U^{D^\bot}$. Let
\begin{equation}
{\nabla}_XY=\widetilde{\nabla}^D_{X}Y+\widetilde{B}(X,Y),~~\widetilde{\nabla}^D_{X}Y=\pi^D {\nabla}_{X}Y,~~\widetilde{B}(X,Y)=\pi^{D^\bot}{\nabla}_XY.
\end{equation}
We call the $\widetilde{B}(X,Y)$ as the second fundamental form with respect to the semi-symmetric metric connection. We have
\begin{equation}
\widetilde{\nabla}^D_{X}Y={\nabla}^{D,L}_{X}Y+X\cdot G(Y,U)-G(X,Y)U^D,~~\widetilde{B}(X,Y)={B}(X,Y)-G(X,Y)U^{D^\bot};
\end{equation}
\begin{equation}
\widetilde{\nabla}^D_{X}(G^D)=0,~~\widetilde{T}^D(X,Y)=-[X,Y]^{D^\bot}+X\cdot G(Y,U)-(-1)^{|X||Y|}Y\cdot G(X,U),
\end{equation}
for all homogeneous $X,Y\in D$. Similar to Theorem 2.9, we have
\begin{thm}
There exists a unique partial linear connection $\widetilde{\nabla}^{D}: D\times D\rightarrow D$ on $D$, which satisfies the property (3.9).
\end{thm}

Let $\xi\in D^{\bot}$ and $X\in D$, then by the definition 2.8, we have
\begin{equation}
{\nabla}_X\xi={\nabla}^L_X\xi+X\cdot G(\xi,U).
\end{equation}
Let $A_\xi:\Gamma(D)\rightarrow \Gamma(D)$ be the shape operator with respect to $\nabla^L$ defined by
\begin{equation}
G^{D^\bot}(B(X,Y),\xi)=(-1)^{|X|(|Y|+|\xi|)}G^D(Y,A_{\xi}X).
\end{equation}
Then $\pi^D{\nabla}^L_X\xi=-(-1)^{|X||\xi|}A_{\xi}X$. Let $A_X\xi=(-1)^{|X||\xi|}A_{\xi}X$, then
\begin{equation}
A_{\alpha X}\xi=\alpha A_X \xi;~~
A_X(\alpha \xi)=(-1)^{|X||\alpha|}\alpha A_X\xi£¬
\end{equation}
\begin{equation}
G^{D^\bot}(B(X,Y),\xi)=(-1)^{|X||Y|}G^D(Y,A_X{\xi}),~~
{\nabla}^L_X\xi=-A_X{\xi}+\nabla^\bot_X\xi,
\end{equation}
\begin{equation}
{\nabla}_X\xi=-\widetilde{A}_X{\xi}+\nabla^\bot_X\xi;~~
\widetilde{A}_X{\xi}=A_X{\xi}-X\cdot G(\xi,U)
\end{equation}
which we called the Weingarten formula with respect to $\nabla$ and $\nabla^{\bot}_X\xi:D\times D^{\bot}\rightarrow D^{\bot}$ is a metric connection on $D^{\bot}$ along $D$.
Given $X_1,X_2,X_3\in D$, the curvature tensor $\widetilde{R}^D$ on $D$ with respect to $\widetilde{\nabla}^D$ is defined by
\begin{equation}
\widetilde{R}^D(X_1,X_2)X_3:=\widetilde{\nabla}^D_{X_1}\widetilde{\nabla}^D_{X_2}X_3-\widetilde{\nabla}^D_{X_2}
\widetilde{\nabla}^D_{X_1}X_3-\widetilde{\nabla}^D_{[X_1,X_2]^D}X_3-\pi^D[[X_1,X_2]^{D^{\bot}},X_3].
\end{equation}
In (2.13), $\widetilde{R}^D$ is a tensor field by adding the extra term $-\pi^D[[X_1,X_2]^{D^{\bot}},X_3]$.
Given $X_1,X_2,X_3,X_4\in D$, the Riemannian curvature tensor $\widetilde{R}^D$ is defined by
$\widetilde{R}^D(X_1,X_2,X_3,X_4)=G^D(\widetilde{R}^D(X_1,X_2)X_3,X_4).$

\begin{thm}
Given $X,Y,Z,W\in D$, we have
\begin{align}
&{R}(X,Y,Z,W)=\widetilde{R}^D(X,Y,Z,W)-(-1)^{(|Y|+|Z|)|W|}g(B(X,W),B(Y,Z))\\
&+(-1)^{|X||Y|}(-1)^{(|X|+|Z|)|W|}g(B(Y,W),B(X,Z))
+(-1)^{(|Y|+|Z|)|W|}G(B(X,W),U)G(Y,Z)\notag\\
&-(-1)^{|X||Y|}(-1)^{(|X|+|Z|)|W|}G(B(Y,W),U)G(X,Z)
+(-1)^{|X|(|Y|+|Z|)}G(B(Y,Z),U)G(X,W)\notag\\
&-(-1)^{|Y||Z|}G(B(X,Z),U)G(Y,W)
-(-1)^{|X|(|Y|+|Z|)}G(Y,Z)G(U^{D^\bot},U^{D^\bot})G(X,W)\notag\\
&+(-1)^{|Y||Z|}G(X,Z)G(U^{D^\bot},U^{D^\bot})G(Y,W)+G([X,Y],B(Z,W)).
\notag
\end{align}
Here Equation (3.16) is called the Gauss equation for $D$ with respect to ${\nabla}$.
\end{thm}
\begin{proof}
From Equations (3.7) and (3.14), we have for $X,Y,Z\in D$
\begin{align}
{\nabla}_X{\nabla}_YZ&={\nabla}_X(\widetilde{\nabla}_Y^DZ)+{\nabla}_X(\widetilde{B}(Y,Z))\\\notag
&=\widetilde{\nabla}^D_X\widetilde{\nabla}^D_YZ+\widetilde{B}(X,\widetilde{\nabla}^D_YZ)\\\notag
&-A_X{\widetilde{B}(Y,Z)}+XG(\widetilde{B}(Y,Z),U)+\nabla^{\bot}_X(\widetilde{B}(Y,Z))\notag,
\end{align}
\begin{align}
{\nabla}_Y{\nabla}_XZ&=\widetilde{\nabla}^D_Y\widetilde{\nabla}^D_XZ+\widetilde{B}(Y,\widetilde{\nabla}^D_XZ)\\
&-A_Y{\widetilde{B}(X,Z)}+YG(\widetilde{B}(X,Z),U)+\nabla^{\bot}_Y(\widetilde{B}(X,Z))\notag,
\end{align}
By the definition 2.8 and $\nabla^L$ with zero torsion, we have
for $X_1,X_2\in {\rm Der}(\Omega(M))$,
\begin{align}
{\nabla}_{X_1}X_2=(-1)^{|X_1||X_2|}{\nabla}_{X_2}X_1+[X_1,X_2]+X_1G(X_2,U)-(-1)^{|X_1||X_2|}X_2G(X_1,U).
\end{align}
So by (3.14) and (3.19), we get
\begin{align}
{\nabla}_{[X,Y]^{D^{\bot}}}Z&=(-1)^{(|X|+|Y|)|Z|}{\nabla}_{Z}([X,Y]^{D^{\bot}})
+[[X,Y]^{D^{\bot}},Z]\\
&+[X,Y]^{D^{\bot}}G(Z,U)
-(-1)^{(|X|+|Y|)|Z|}ZG([X,Y]^{D^{\bot}},U)\notag\\
&=-(-1)^{(|X|+|Y|)|Z|}A_Z({[X,Y]^{D^{\bot}}})+(-1)^{(|X|+|Y|)|Z|}\nabla^{\bot}_Z([X,Y]^{D^{\bot}})\notag\\
&+{[X,Y]^{D^{\bot}}}G(Z,U)+[{[X,Y]^{D^{\bot}}},Z].\notag
\end{align}
\noindent By ${\nabla}_{[X,Y]}Z={\nabla}_{[X,Y]^{D}}Z+{\nabla}_{[X,Y]^{D^{\bot}}}Z$ and (3.20) and (3.7), we have
\begin{align}
&{\nabla}_{[X,Y]}Z=\widetilde{\nabla}^D_{[X,Y]^{D}}Z+\widetilde{B}([X,Y]^{D},Z)\\
&-(-1)^{(|X|+|Y|)|Z|}A_Z({[X,Y]^{D^{\bot}}})+(-1)^{(|X|+|Y|)|Z|}\nabla^{\bot}_Z([X,Y]^{D^{\bot}})\notag\\
&+{[X,Y]^{D^{\bot}}}G(Z,U)+[{[X,Y]^{D^{\bot}}},Z].\notag
\end{align}
By (3.15),(3.17),(3.18) and (3.21), we have
\begin{align}
{R}(X,Y)Z=&\widetilde{R}^D(X,Y)Z-\pi^{D^{\bot}}[{[X,Y]^{D^{\bot}}},Z]+\widetilde{B}(X,\widetilde{\nabla}^D_YZ)\\
&-(-1)^{|X||Y|}\widetilde{B}(Y,\widetilde{\nabla}^D_XZ)
-\widetilde{B}([X,Y]^{D},Z)-A_X({\widetilde{B}(Y,Z)})\notag\\
&+(-1)^{|X||Y|}A_Y({\widetilde{B}(X,Z)})
+\nabla^{\bot}_X(\widetilde{B}(Y,Z))-(-1)^{|X||Y|}\nabla^{\bot}_Y(\widetilde{B}(X,Z))\notag\\
&+XG(\widetilde{B}(Y,Z),U)-(-1)^{|X||Y|}YG(\widetilde{B}(X,Z),U)
+(-1)^{(|X|+|Y|)|Z|}A_Z({[X,Y]^{D^{\bot}}})\notag\\
&-(-1)^{(|X|+|Y|)|Z|}\nabla^{\bot}_Z([X,Y]^{D^{\bot}})-{[X,Y]^{D^{\bot}}}G(Z,U).\notag
\end{align}
By (3.8),(3.13),(3.22), we get (3.16).
\end{proof}

\begin{cor} If $U=0$, then ${\nabla}=\nabla^L$, and
given $X,Y,Z,W\in D$, we have
\begin{align}
&{R}^L(X,Y,Z,W)={R}^D(X,Y,Z,W)-(-1)^{(|Y|+|Z|)|W|}g(B(X,W),B(Y,Z))\\
&+(-1)^{|X||Y|}(-1)^{(|X|+|Z|)|W|}g(B(Y,W),B(X,Z))
+G([X,Y],B(Z,W)).
\notag
\end{align}
\end{cor}
\begin{thm}
Given $X,Y,Z\in\Gamma(D)$, we have
\begin{align}
({R}(X,Y)Z)^{D^{\bot}}=&(\nabla^{\bot}_X\widetilde{B})(Y,Z)-(-1)^{|X||Y|}(\nabla^{\bot}_Y\widetilde{B})(X,Z)\\
&-G(X,U)\widetilde{B}(Y,Z)+(-1)^{|X||Y|}G(Y,U)\widetilde{B}(X,Z)-\pi^{D^{\bot}}[{[X,Y]^{D^{\bot}}},Z]\notag\\
&-(-1)^{(|X|+|Y|)|Z|}\nabla^{\bot}_Z([X,Y]^{D^{\bot}})-{[X,Y]^{D^{\bot}}}G(Z,U),\notag
\end{align}
where $(\nabla^{\bot}_X\widetilde{B})(Y,Z)=\nabla^{\bot}_X(\widetilde{B}(Y,Z))-\widetilde{B}(\widetilde{\nabla}^D_XY,Z)
-(-1)^{|X||Y|}\widetilde{B}(Y,\widetilde{\nabla}^D_XZ).$ Equation (3.24) is called the Codazzi
equation with respect to ${\nabla}$.
\end{thm}
\begin{proof}
From (3.22), we have
\begin{align}
(\widetilde{R}(X,Y)Z)^{D^{\bot}}=&-\pi^{D^{\bot}}[{[X,Y]^{D^{\bot}}},Z]+\widetilde{B}(X,\widetilde{\nabla}^D_YZ)\\
&-(-1)^{|X||Y|}\widetilde{B}(Y,\widetilde{\nabla}^D_XZ)
-\widetilde{B}([X,Y]^{D},Z)
+\nabla^{\bot}_X(\widetilde{B}(Y,Z))\notag\\
&-(-1)^{|X||Y|}\nabla^{\bot}_Y(\widetilde{B}(X,Z))
-(-1)^{(|X|+|Y|)|Z|}\nabla^{\bot}_Z([X,Y]^{D^{\bot}})-{[X,Y]^{D^{\bot}}}G(Z,U).\notag
\end{align}
By (3.19), we have for $X,Y\in D$,
\begin{align}
[X,Y]^D=\widetilde{\nabla}^D_XY-(-1)^{|X||Y|}\widetilde{\nabla}^D_YX-XG(Y,U)+(-1)^{|X||Y|}YG(X,U).
 \end{align}
 By (3.26) and the definition of $(\nabla^{\bot}_X\widetilde{B})(Y,Z)$ and (3.25), we get (3.24).
\end{proof}
\begin{cor} If $U=0$, then we have
\begin{align}
({R^L}(X,Y)Z)^{D^{\bot}}=&(\nabla^{\bot}_X{B})(Y,Z)-(-1)^{|X||Y|}(\nabla^{\bot}_Y{B})(X,Z)\\
&-\pi^{D^{\bot}}[{[X,Y]^{D^{\bot}}},Z]-(-1)^{(|X|+|Y|)|Z|}\nabla^{\bot}_Z([X,Y]^{D^{\bot}}).\notag
\end{align}
\end{cor}
\begin{thm}
Given $X,Y\in D$, $\xi\in D^{\bot}$, we have
\begin{align}
({R}(X,Y)\xi)^{D^{\bot}}=-\widetilde{B}(X,\widetilde{A}_Y{\xi})+(-1)^{|X||Y|}\widetilde{B}(Y,\widetilde{A}_X{\xi})
+\widetilde{R}^{L^{\bot}}(X,Y)\xi
\end{align}
where
\begin{align}
\widetilde{R}^{L^{\bot}}(X,Y)\xi:=\nabla^{\bot}_X\nabla^{\bot}_Y\xi-(-1)^{|X||Y|}\nabla^{\bot}_Y\nabla^{\bot}_X\xi-\nabla^{\bot}_{[X,Y]^D}\xi
-\pi^{D^{\bot}}\widetilde{\nabla}_{[X,Y]^{\bot}}\xi.
\end{align}
 Equation (3.28) is called the Ricci
equation for $D$ with respect to ${\nabla}$.
\end{thm}
\begin{proof}
From (3.7) and (3.14), we have
\begin{align}
{\nabla}_X{\nabla}_Y\xi
&=-\widetilde{\nabla}^D_X(\widetilde{A}_Y{\xi})-\widetilde{B}(X,\widetilde{A}_Y{\xi})-\widetilde{A}_X{\nabla^{\bot}_Y\xi}
+\nabla^{\bot}_X\nabla^{\bot}_Y\xi,
\end{align}
\begin{align}
{\nabla}_Y{\nabla}_X\xi=-\widetilde{\nabla}^D_Y(\widetilde{A}_X{\xi})-\widetilde{B}(Y,\widetilde{A}_X{\xi})-
\widetilde{A}_Y{\nabla^{\bot}_X\xi}+\nabla^{\bot}_Y\nabla^{\bot}_X\xi,
\end{align}
\begin{align}
{\nabla}_{[X,Y]}\xi
=&-\widetilde{A}_{[X,Y]^D}{\xi}+\nabla^{\bot}_{[X,Y]^D}\xi+\pi^D\widetilde{\nabla}_{[X,Y]^{D^{\bot}}}\xi
+\pi^{D^{\bot}}\widetilde{\nabla}_{[X,Y]^{D^{\bot}}}\xi.
\end{align}
From (3.29)-(3.32), we get (3.28).
\end{proof}
\begin{cor} If $U=0$, then we have
\begin{align}
({R}^L(X,Y)\xi)^{D^{\bot}}=-{B}(X,{A}_{\xi}Y)+(-1)^{|X||Y|}{B}(Y,{A}_{\xi}X)+{R}^{L^{\bot}}(X,Y)\xi,
\end{align}
where
\begin{align}
{R}^{L^{\bot}}(X,Y)\xi:=\nabla^{\bot}_X\nabla^{\bot}_Y\xi-(-1)^{|X||Y|}\nabla^{\bot}_Y\nabla^{\bot}_X\xi-
\nabla^{\bot}_{[X,Y]^D}\xi-\pi^{D^{\bot}}{\nabla}^L_{[X,Y]^{\bot}}\xi.
\end{align}
\end{cor}

\section{The Lie derivative of connections for distributions on the algebra of differential forms}
 \indent The setup in this section is same as the section 3.
 \begin{defn}Suppose $X\in D$ and $L_X:D\rightarrow D;~~Y\mapsto [X,Y]^D$. The mapping
 \begin{align}
&L_X(\nabla^{D,L}):D\times D\rightarrow D,\\
&L_X(\nabla^{D,L})(Y,Z):=L_X(\nabla^{D,L}_YZ)-\nabla^{D,L}_{[X,Y]^D}Z-(-1)^{|X||Y|}\nabla^{D,L}_Y([X,Z]^D),\notag
\end{align}
 is called the Lie derivative of the partial connection $\nabla^{D,L}$ along the operator $X\in D$. The mapping
 \begin{align}
&L_X(\widetilde{\nabla}^{D}):D\times D\rightarrow D,\\
&L_X(\widetilde{\nabla}^{D})(Y,Z):=L_X(\widetilde{\nabla}^{D}_YZ)-\widetilde{\nabla}^{D}_{[X,Y]^D}Z
-(-1)^{|X||Y|}\widetilde{\nabla}^{D}_Y([X,Z]^D),\notag
\end{align}
 is called the Lie derivative of the partial connection $\widetilde{\nabla}^{D}$ along the operator $X\in D$
 \end{defn}
 \begin{prop}For $X,Y,Z,W\in D$, then we have
  \begin{align}
&[L_X,L_Y]({\nabla}^{D,L})(Z,W):=[L_X,L_Y]({\nabla}^{D,L}_ZW)+(-1)^{|X||Y|}{\nabla}^{D,L}_{[Y,[X,Z]^D]^D}W\\
&+(-1)^{|X||Y|+|X||Z|+|Y||Z|}{\nabla}^{D,L}_Z[Y,[X,W]^D]^D-{\nabla}^{D,L}_{[X,[Y,Z]^D]^D}W\notag\\
&-(-1)^{(|X|+|Y|)|Z|}{\nabla}^{D,L}_Z[X,[Y,W]^D]^D.\notag
\end{align}
(4.3) is also correct for $\widetilde{\nabla}^{D}$.
 \end{prop}
 \begin{proof}
By the definition 4.1, we have
 \begin{align}
&L_X(L_Y({\nabla}^{D,L}))(Z,W):=L_XL_Y({\nabla}^{D,L}_ZW)-L_X(\nabla^{D,L}_{[Y,Z]^D}W)\\
&-(-1)^{|Y||Z|}L_X(\nabla^{D,L}_Z([Y,W]^D))
-(-1)^{|X||Y|}L_Y(\nabla^{D,L}_{[X,Z]^D}W)\notag\\
&-(-1)^{|X|(|Y|+|Z|)}L_Y(\nabla^{D,L}_Z([X,W]^D))
+(-1)^{|X||Y|}\nabla^{D,L}_{[Y,[X,Z]^D]^D}W\notag\\
&+(-1)^{|Y||Z|}\nabla^{D,L}_{[X,Z]^D}[Y,W]^D+(-1)^{|X|(|Y|+|Z|)}\nabla^{D,L}_{[Y,Z]^D}[X,W]^D\notag\\
&+(-1)^{|X||Y|+|X||Z|+|Y||Z|}{\nabla}^{D,L}_Z[Y,[X,W]^D]^D.\notag
\end{align}
By (4.4) and $[L_X,L_Y]({\nabla}^{D,L})=L_X(L_Y({\nabla}^{D,L}))-(-1)^{|X||Y|}L_Y(L_X({\nabla}^{D,L}))$, we get (4.3).
\end{proof}
\begin{cor} If $D$ is integrable , then we have
\begin{align}
&[L_X,L_Y]({\nabla}^{D,L})=L_{[X,Y]}({\nabla}^{D,L}).
\end{align}
 \end{cor}
 \begin{proof}
By the graded Jacobi identity:
 \begin{align}
&[X,[Y,Z]]-[[X,Y],Z]-(-1)^{|X||Y|}[Y,[X,Z]]=0,
\end{align}
and $D$ is integrable, we have
\begin{align}
&[L_X,L_Y]({\nabla}^{D,L}_ZW)=L_{[X,Y]}({\nabla}^{D,L}_ZW),\\
&
(-1)^{|X||Y|}{\nabla}^{D,L}_{[Y,[X,Z]^D]^D}W
-{\nabla}^{D,L}_{[X,[Y,Z]^D]^D}W=-{\nabla}^{D,L}_{[[X,Y],Z]}W\notag\\
&(-1)^{|X||Y|+|X||Z|+|Y||Z|}{\nabla}^{D,L}_Z[Y,[X,W]^D]^D
-(-1)^{(|X|+|Y|)|Z|}{\nabla}^{D,L}_Z[X,[Y,W]^D]^D\notag\\
&=-(-1)^{(|X|+|Y|)|Z|}{\nabla}^{D,L}_Z[[X,Y],W],\notag
\end{align}
By (4.3) and (4.7), we get (4.5).
\end{proof}
 \begin{defn}Suppose $X,Y,Z,W\in D$. The mapping
 \begin{align}
&L_X(R^{D,L}):D\times D\times D\rightarrow D; (Y,Z,W)\rightarrow (L_X(R^{D,L}))(Y,Z,W),\\
&(L_X(R^{D,L}))(Y,Z,W):=L_X(R^{D,L}(Y,Z,W))-R^{D,L}(L_XY,Z,W)\notag\\
&-(-1)^{|X||Y|}R^{D,L}(Y,L_XZ,W)-(-1)^{|X|(|Y|+|Z|)}R^{D,L}(Y,Z,L_XW),\notag
\end{align}
 is called the Lie derivative of the curvature $R^{D,L}$ along the operator $X\in D$.
 \end{defn}
\begin{thm}Let $D$ be integrable and $X,Y,Z,W\in D$, we have
\begin{align}
&(L_X(R^{D,L}))(Y,Z,W)=-(L_X\nabla^{D,L})([Y,Z],W)+(L_X\nabla^{D,L})(Y,\nabla^{D,L}_ZW)\\
&+(-1)^{|X||Y|}\nabla^{D,L}_Y[(L_X\nabla^{D,L})(Z,W)]-(-1)^{|Y||Z|}(L_X\nabla^{D,L})(Z,\nabla^{D,L}_YW)\notag\\
&-(-1)^{(|X|+|Y|)|Z|}\nabla^{D,L}_Z[(L_X\nabla^{D,L})(Y,W)].\notag
\end{align}
\end{thm}
\begin{proof}
Since $D$ is integrable, by the definition of $R^{D,L}$, we have
\begin{align}
&L_X(R^{D,L}(Y,Z,W))=L_X(\nabla^{D,L}_Y\nabla^{D,L}_ZW)-(-1)^{|Y||Z|}L_X(\nabla^{D,L}_Z\nabla^{D,L}_YW)-L_X(\nabla^{D,L}_{[Y,Z]}W);\\
&R^{D,L}(L_XY,Z,W)=\nabla^{D,L}_{[X,Y]}\nabla^{D,L}_ZU-(-1)^{(|X|+|Y|)|Z|}\nabla^{D,L}_Z\nabla^{D,L}_{[X,Y]}W-\nabla^{D,L}_{[[X,Y],Z]}W;\notag\\
&R^{D,L}(Y,L_XZ,W)=\nabla^{D,L}_Y\nabla^{D,L}_{[X,Z]}W-(-1)^{|Y|(|X|+|Z|)}\nabla^{D,L}_{[X,Z]}\nabla^{D,L}_YW-\nabla^{D,L}_{[Y,[X,Z]]}W;\notag\\
&R^{D,L}(Y,Z,L_XW)=\nabla^{D,L}_Y\nabla^{D,L}_Z(L_XW)-(-1)^{|Y||Z|}\nabla^{D,L}_Z\nabla^{D,L}_Y(L_XW)-\nabla^{D,L}_{[Y,Z]}(L_XW).\notag
\end{align}
We also have
\begin{align}
&-L_X(\nabla^{D,L}_{[Y,Z]}W)+\nabla^{D,L}_{[[X,Y],Z]}W+(-1)^{|X||Y|}\nabla^{D,L}_{[Y,[X,Z]]}W\\
&+(-1)^{(|X|+|Y|)|Z|}\nabla^{D,L}_{[Y,Z]}(L_XW)=-(L_X\nabla^{D,L})([Y,Z],U),\notag\\
&L_X(\nabla^{D,L}_Y\nabla^{D,L}_ZW)-\nabla^{D,L}_{[X,Y]}\nabla^{D,L}_ZU\notag\\
&-(-1)^{|X||Y|}\nabla^{D,L}_Y\nabla^{D,L}_{[X,Z]}W
-(-1)^{(|X|+|Y|)|Z|}\nabla^{D,L}_Y\nabla^{D,L}_Z(L_XW)\notag\\
&=(L_X\nabla^{D,L})(Y,\nabla^{D,L}_ZW)
+(-1)^{|X||Y|}\nabla^{D,L}_Y[(L_X\nabla^{D,L})(Z,W)],\notag\\
&-(-1)^{|Y||Z|}L_X(\nabla^{D,L}_Z\nabla^{D,L}_YW)+(-1)^{(|X|+|Y|)|Z|}\nabla^{D,L}_Z\nabla^{D,L}_{[X,Y]}W\notag\\
&
+(-1)^{|X||Y|}(-1)^{|Y|(|X|+|Z|)}\nabla^{D,L}_{[X,Z]}\nabla^{D,L}_YW+(-1)^{(|X|+|Y|)|Z|}(-1)^{|Y||Z|}\nabla^{D,L}_Z\nabla^{D,L}_Y(L_XW)\notag\\
&
=-(-1)^{|Y||Z|}(L_X\nabla^{D,L})(Z,\nabla^{D,L}_YW)
-(-1)^{(|X|+|Y|)|Z|}\nabla^{D,L}_Z[(L_X\nabla^{D,L})(Y,W)].\notag
\end{align}
By (4.10) and (4.11), we get (4.9).
\end{proof}
\begin{defn}Suppose $X\in D$ and $L^\bot_X:D^\bot\rightarrow D^\bot;~~N\mapsto [X,N]^{D^\bot}$. The mapping
 \begin{align}
&L^\bot_X(\nabla^{\bot}):D\times D^\bot\rightarrow D^\bot,\\
&L^\bot_X(\nabla^{\bot})(Y,N):=L^\bot_X(\nabla^{\bot}_YN)-\nabla^{\bot}_{[X,Y]^D}N-(-1)^{|X||Y|}\nabla^{\bot}_Y([X,N]^{D^\bot}),\notag
\end{align}
 is called the Lie derivative of the normal connection $\nabla^{\bot}$ along the operator $X\in D$. Then
 \begin{align}
&L^\bot_X(\nabla^{\bot})(\alpha Y,N)=(-1)^{|X||\alpha|}\alpha L^\bot_X(\nabla^{\bot})(Y,N),\\
&L^\bot_X(\nabla^{\bot})(Y,\alpha N)=[X,Y]^{D^\bot}(\alpha)N+(-1)^{(|X|+|Y|)|\alpha|}\alpha L^\bot_X(\nabla^{\bot})(Y,N).\notag
\end{align}
 \end{defn}
 Similar to Proposition 4.2, we have
 \begin{prop}For $X,Y,Z\in D$ and $N\in D^\bot$, then we have
  \begin{align}
&[L^\bot_X,L^\bot_Y]({\nabla}^{\bot})(Z,N):=[L^\bot_X,L^\bot_Y]({\nabla}^{\bot}_ZN)+(-1)^{|X||Y|}{\nabla}^{\bot}_{[Y,[X,Z]^D]^D}N\\
&+(-1)^{|X||Y|+|X||Z|+|Y||Z|}{\nabla}^{\bot}_Z(L^\bot_YL^\bot_XN)-{\nabla}^{\bot}_{[X,[Y,Z]^D]^D}N\notag\\
&-(-1)^{(|X|+|Y|)|Z|}{\nabla}^{\bot}_Z(L^\bot_XL^\bot_YN).\notag
\end{align}
In particular, when $D$ is integrable, we have
\begin{align}
&[L^\bot_X,L^\bot_Y]({\nabla}^{\bot})=L^\bot_{[X,Y]}({\nabla}^{\bot}).
\end{align}
 \end{prop}
\begin{defn}Suppose $X,Y,Z\in D$, $N\in D^\bot$ and $D$ is integrable. The mapping
 \begin{align}
&L^\bot_X(R^{\bot}):D\times D\times D^\bot\rightarrow D^\bot; (Y,Z,N)\rightarrow (L^\bot_X(R^{\bot}))(Y,Z,N),\\
&(L^\bot_X(R^{\bot}))(Y,Z,N):=L^\bot_X(R^{\bot}(Y,Z)N)-R^{\bot}(L_XY,Z,N)\notag\\
&-(-1)^{|X||Y|}R^{\bot}(Y,L_XZ,N)-(-1)^{|X|(|Y|+|Z|)}R^{\bot}(Y,Z,L^\bot_XN),\notag
\end{align}
 is called the Lie derivative of the normal curvature $R^{\bot}$ along the operator $X\in D$.
 \end{defn}
 Similar to Theorem 4.5, we have
\begin{thm}Let $D$ be integrable and $X,Y,Z\in D$, $N\in D^\bot$, we have
\begin{align}
&(L^\bot_X(R^{\bot}))(Y,Z,N):=-(L^\bot_X\nabla^{\bot})([Y,Z],N)+(L^\bot_X\nabla^{\bot})(Y,\nabla^{\bot}_ZN)\\
&+(-1)^{|X||Y|}\nabla^{\bot}_Y[(L^\bot_X\nabla^{\bot})(Z,N)]-(-1)^{|Y||Z|}(L^\bot_X\nabla^{\bot})(Z,\nabla^{\bot}_YN)\notag\\
&-(-1)^{(|X|+|Y|)|Z|}\nabla^{\bot}_Z[(L^\bot_X\nabla^{\bot})(Y,N)].\notag
\end{align}
\end{thm}

\section{Several other affine connections on the algebra of differential forms}
 \indent In this section, we define several other affine connections on the algebra of differential forms.
 Firstly, we define the canonical connection.
Let $(M,\overline{X_1},\cdots,\overline{X_m})$ be a parallelizable manifold (see the definition 1.1 in \cite{YE}). Then
${\rm Der}(\Omega(M))=\{L_{\overline{X_1}},\cdots,L_{\overline{X_m}},i_{\overline{X_1}},\cdots,i_{\overline{X_m}}\}$.
\begin{prop}On a parallelizable manifold $(M,\overline{X_1},\cdots,\overline{X_m},{\rm Der}(\Omega(M)))$, there exists a unique linear connection
$\nabla^c$ for which satisfies $\nabla^c_X(L_{\overline{X_j}})=\nabla^c_X(i_{\overline{X_j}})=0$ for $1\leq j\leq m$ and
$X\in {\rm Der}(\Omega(M))$.
\end{prop}
\begin{proof}Firstly we prove the uniqueness. Assume that $\nabla$ is a linear connection satisfies
$\nabla^c_X(L_{\overline{X_j}})=\nabla^c_X(i_{\overline{X_j}})=0$. Let $Y=\sum_{j=1}^m\omega_j L_{\overline{X_j}}+
\sum_{j=1}^m\omega'_j i_{\overline{X_j}}$ for $\omega_j,\omega'_j\in \Omega(M)$. Then the connection $\nabla$ is uniquely determined by
 \begin{align}
&\nabla_XY=\sum_{j=1}^mX(\omega_j) L_{\overline{X_j}}+
\sum_{j=1}^mX(\omega'_j) i_{\overline{X_j}}.
\end{align}
To prove the existence, let $\nabla^c:{\rm Der}(\Omega(M))\times {\rm Der}(\Omega(M))\rightarrow {\rm Der}(\Omega(M))$ be defined by (5.1). We may
show that $\nabla^c$ is a linear connection.
\end{proof}
\begin{prop}The torsion tensor $T^c$ of the canonical connection is given by
\begin{align}
&T^c(L_{\overline{X_j}},L_{\overline{X_l}})=-L_{[\overline{X_j},{\overline{X_l}}]};~~
T^c(L_{\overline{X_j}},i_{\overline{X_l}})=-i_{[\overline{X_j},{\overline{X_l}}]};\notag\\
&
T^c(i_{\overline{X_j}},L_{\overline{X_l}})=i_{[\overline{X_l},{\overline{X_j}}]};~~
T^c(i_{\overline{X_j}},i_{\overline{X_l}})=0.
\end{align}
The canonical connection is flat. The canonical connection is a metric connection about $G_g$ where $g$ is a metric on $M$ such that
$\{\overline{X_1},\cdots,\overline{X_m}\}$ is an orthonormal basis about $g$.
\end{prop}
\begin{defn}
The dual connection $\widetilde{\nabla}$ is defined by $\widetilde{\nabla}_XY:=(-1)^{|X||Y|}\nabla^c_YX+[X,Y]$. Then
$\widetilde{T}(X,Y)=-T^c(X,Y).$ Let $\nabla^\lambda_XY=(1-\lambda)\nabla^c_XY+\lambda\widetilde{\nabla}_XY$ for $\lambda$ is a constant.
\end{defn}
\begin{lem}
The following equalities hold
\begin{align}
&\nabla^\lambda_{L_{\overline{X_j}}}L_{\overline{X_l}}=\lambda L_{[\overline{X_j},{\overline{X_l}}]};~~
\nabla^\lambda_{L_{\overline{X_j}}}i_{\overline{X_l}}=\lambda i_{[\overline{X_j},{\overline{X_l}}]};~~\\
&\nabla^\lambda_{i_{\overline{X_j}}}L_{\overline{X_l}}=\lambda i_{[\overline{X_j},{\overline{X_l}}]};~~
\nabla^\lambda_{i_{\overline{X_j}}}i_{\overline{X_l}}=0.\notag
\end{align}
\end{lem}
By Lemma 5.4, we have
\begin{lem}
The following equalities hold
\begin{align}
&R^\lambda(L_{\overline{X_j}},L_{\overline{X_k}})L_{\overline{X_l}}=(\lambda^2-\lambda)L_{[[\overline{X_j},\overline{X_k}],\overline{X_l}]}
+(\lambda-\lambda^2)L_{\Sigma_\mu[\overline{X_j}(C^\mu_{kl})-\overline{X_k}(C^\mu_{jl})]\overline{X_\mu}}-\lambda L_{\Sigma_\mu
\overline{X_l}(C^\mu_{jk})
\overline{X_\mu}};\\
&R^\lambda(L_{\overline{X_j}},L_{\overline{X_k}})i_{\overline{X_l}}=
(\lambda-\lambda^2)i_{\Sigma_\mu[\overline{X_j}(C^\mu_{kl})-\overline{X_k}(C^\mu_{jl})]\overline{X_\mu}}
+(\lambda^2-\lambda)i_{[[\overline{X_j},\overline{X_k}],\overline{X_l}]}
-\lambda i_{\Sigma_\mu
\overline{X_l}(C^\mu_{jk})
\overline{X_\mu}};\notag\\
&R^\lambda(L_{\overline{X_j}},i_{\overline{X_k}})L_{\overline{X_l}}=(\lambda^2-\lambda)i_{[[\overline{X_j},\overline{X_k}],\overline{X_l}]}
+(\lambda-\lambda^2)i_{\Sigma_\mu[\overline{X_j}(C^\mu_{kl})-\overline{X_k}(C^\mu_{jl})]\overline{X_\mu}}
-\lambda i_{\Sigma_\mu
\overline{X_l}(C^\mu_{jk})
\overline{X_\mu}};\notag\\
&R^\lambda(L_{\overline{X_j}},i_{\overline{X_k}})i_{\overline{X_l}}=0;~~
R^\lambda(i_{\overline{X_j}},i_{\overline{X_k}})L_{\overline{X_l}}=0;~~
R^\lambda(i_{\overline{X_j}},i_{\overline{X_k}})i_{\overline{X_l}}=0,\notag
\end{align}
where $[\overline{X_k},\overline{X_l}]=\Sigma_{\mu=1}^mC^\mu_{kl}\overline{X_\mu}$.
\end{lem}
By the definition 2.7 and Lemma 5.5, we get
\begin{thm}
$\nabla^\lambda$ is Ricci flat.
\end{thm}
Let $\nabla^\omega_XY=(1-\omega)\nabla^c_XY+\omega\widetilde{\nabla}_XY$ for $\omega\in \Omega^{\rm even}(M)$. Then $\nabla^\omega$ is a connection.
Similar to Lemma 5.5, we have
\begin{lem}
If $C^\mu_{kl}$ is a constant, then the following equalities hold
\begin{align}
&R^\omega(L_{\overline{X_j}},L_{\overline{X_k}})L_{\overline{X_l}}=L_{\overline{X_j}}(\omega)L_{[\overline{X_k},\overline{X_l}]}-
L_{\overline{X_k}}(\omega)L_{[\overline{X_j},\overline{X_l}]}
+(\omega^2-\omega)L_{[[\overline{X_j},\overline{X_k}],\overline{X_l}]}
;\\
&R^\omega(L_{\overline{X_j}},L_{\overline{X_k}})i_{\overline{X_l}}=
L_{\overline{X_j}}(\omega)i_{[\overline{X_k},\overline{X_l}]}-
L_{\overline{X_k}}(\omega)i_{[\overline{X_j},\overline{X_l}]}
+(\omega^2-\omega)i_{[[\overline{X_j},\overline{X_k}],\overline{X_l}]}
;\notag\\
&R^\omega(L_{\overline{X_j}},i_{\overline{X_k}})L_{\overline{X_l}}=
L_{\overline{X_j}}(\omega)i_{[\overline{X_k},\overline{X_l}]}-
i_{\overline{X_k}}(\omega)L_{[\overline{X_j},\overline{X_l}]}
+(\omega^2-\omega)i_{[[\overline{X_j},\overline{X_k}],\overline{X_l}]}
;\notag\\
&R^\omega(L_{\overline{X_j}},i_{\overline{X_k}})i_{\overline{X_l}}=-i_{\overline{X_k}}(\omega)i_{[\overline{X_j},\overline{X_l}]};~~\notag\\
&R^\omega(i_{\overline{X_j}},i_{\overline{X_k}})L_{\overline{X_l}}=i_{\overline{X_j}}(\omega)i_{[\overline{X_k},\overline{X_l}]}
+i_{\overline{X_k}}(\omega)i_{[\overline{X_j},\overline{X_l}]}
;~~
R^\omega(i_{\overline{X_j}},i_{\overline{X_k}})i_{\overline{X_l}}=0.\notag
\end{align}
\end{lem}

By the definition 2.7 and Lemma 5.7, we get
\begin{thm}
If $C^\mu_{kl}$ is a constant, then the following equalities hold
\begin{align}
&{\rm Ric}^\omega(L_{\overline{X_j}},L_{\overline{X_l}})=L_{[\overline{X_j},\overline{X_l}]}(\omega);~~
{\rm Ric}^\omega(L_{\overline{X_j}},i_{\overline{X_l}})={\rm Ric}^\omega(i_{\overline{X_j}},L_{\overline{X_l}})
=i_{[\overline{X_j},\overline{X_l}]}(\omega);~~
{\rm Ric}^\omega(i_{\overline{X_j}},i_{\overline{X_l}})=0.
\end{align}
In general, $\nabla^\omega$ is non Ricci flat.
\end{thm}
By the definitions of the Lie derivative of connections and $\widetilde{\nabla}$ and the graded Jacobi identity, we have
\begin{prop}
On a parallelizable manifold, we have $(L_X\widetilde{\nabla})(Y,Z)=(-1)^{|Y||Z|}(L_X {\nabla}^c)(Y,Z)$ for $X\in {\rm Der}\Omega(M)$.
\end{prop}
\begin{prop}
On a parallelizable manifold, if $\nabla^cX=0$ for $X\in {\rm Der}\Omega(M)$, then we have
\begin{align}
&(L_X\widetilde{\nabla})(Y,Z)=-(-1)^{|Y||Z|}T^c(X,\nabla^c_ZY)+(-1)^{(|X|+|Y|)|Z|}\nabla^c_Z(T^c(X,Y)).
\end{align}
\end{prop}
\begin{proof}
By the Jacobi identity and $[X,Y]=-T^c(X,Y)+\nabla^c_XY-(-1)^{|X||Y|}\nabla^c_YX$ and $\nabla^cX=0$ and $\nabla^c$ is flat, we have
\begin{align}
(L_X\widetilde{\nabla})(Y,Z)&=L_X(\widetilde{\nabla}_YZ)-\widetilde{\nabla}_{[X,Y]}Z-(-1)^{|X||Y|}\widetilde{\nabla}_Y[X,Z]\\
&=L_X((-1)^{|Y||Z|}\nabla^c_ZY+[Y,Z])-(-1)^{(|X|+|Y|)|Z|}\nabla^c_Z[X,Y]-[[X,Y],Z]\notag\\
&-(-1)^{|X||Y|}(-1)^{|Y|(|X|+|Z|)}\nabla^c_{[X,Z]}Y
-(-1)^{|X||Y|}[Y,[X,Z]]\notag\\
&=(-1)^{|Y||Z|}L_X\nabla^c_ZY-(-1)^{(|X|+|Y|)|Z|}\nabla^c_Z[X,Y]-(-1)^{|Y||Z|}\nabla^c_{[X,Z]}Y\notag\\
&=(-1)^{|Y||Z|}[-T^c(X,\nabla^c_ZY)+\nabla^c_X\nabla^c_ZY-(-1)^{|X||\nabla^c_ZY|}\nabla^c_{\nabla^c_ZY}X\notag\\
&-(-1)^{|X||Z|}\nabla^c_Z(\nabla^c_XY-(-1)^{|X||Y|}\nabla^c_YX-T^c(X,Y))-\nabla^c_{[X,Z]}Y]\notag\\
&=(-1)^{|Y||Z|}[-T^c(X,\nabla^c_ZY)+R^c(X,Z)Y+(-1)^{|X||Z|}\nabla^c_Z(T^c(X,Y))]\notag\\
&=-(-1)^{|Y||Z|}T^c(X,\nabla^c_ZY)+(-1)^{(|X|+|Y|)|Z|}\nabla^c_Z(T^c(X,Y)).\notag
\end{align}
\end{proof}
\vskip 0.5 true cm
In the following, the setup is the same as Section 3. Let ${\rm Der}(\Omega(M))=D\oplus D^\bot.$ and
 $G=G^D\oplus G^{D^\bot}$ where $G^D$ and $ G^{D^\bot}$ are metric on $D$ and $D^\bot$ respectively.
Let $\pi^D:{\rm Der}(\Omega(M))\rightarrow D$, $\pi^{D^\bot}:{\rm Der}(\Omega(M))\rightarrow D^\bot$ be the projections with the even grading.
\begin{defn}
Given an any connection $\nabla$ on ${\rm Der}(\Omega(M))$, we define two linear connections:
the Schouten connection $\nabla^s$ and the Vranceanu connection $\nabla^v$ on ${\rm Der}(\Omega(M))$ respectively
by
\begin{align}
&\nabla^s_XY=\pi^D\nabla_X \pi^DY+\pi^{D^\bot}\nabla_X  \pi^{D^\bot}Y,
\end{align}
\begin{align}
&\nabla^v_XY=\pi^D\nabla_{\pi^D X} \pi^DY+\pi^{D^\bot}\nabla_{\pi^{D^\bot} X } \pi^{D^\bot}Y
+\pi^D[\pi^{D^\bot} X,\pi^DY]+\pi^{D^\bot}[\pi^D X,\pi^{D^\bot}Y].
\end{align}
\end{defn}
The proof of the following proposition is the same as \cite{Ia}.
\begin{prop}
1)Distributions $D,D^\bot$ are both parallel with respect to the connections $\nabla^s$ and $\nabla^v$.\\
2)The connection $\nabla^s$ is equal to the connection $\nabla$ if and only if
 Distributions $D,D^\bot$ are both parallel with respect to the connections $\nabla$.\\
 3)If the connection $\nabla$ is symmetric and $D,D^\bot$ are both integrable, then the connection $\nabla^v$ is symmetric.\\
 4)If the connection $\nabla^s$ or $\nabla^v$ is symmetric, then $D,D^\bot$ are both integrable.
\end{prop}

\section{Acknowledgements}

The author was supported in part by  NSFC No.11771070.

\vskip 0.5 true cm


\bigskip

\noindent {\footnotesize {\it Yong Wang} \\
{School of Mathematics and Statistics, Northeast Normal University, Changchun 130024, China}\\
{Email: wangy581@nenu.edu.cn}

\end{document}